\documentclass{article}
\usepackage[utf8]{inputenc}

\usepackage{amsmath,amssymb,amsthm}
\usepackage{arydshln,hyperref,graphicx}

\newcommand*{\QEDA}{\hfill\ensuremath{\blacksquare}}

\newtheorem{thm}{Theorem}[section]
\newtheorem{lem}[thm]{Lemma}
\newtheorem{cor}[thm]{Corollary}
\newtheorem{obs}[thm]{Observation}

\theoremstyle{definition}
\newtheorem{defn}[thm]{Definition}
\newtheorem*{exa}{Example}

\title{The principal Erd\H{o}s--Gallai differences of a degree sequence}
\author{Michael D. Barrus\\Department of Mathematics and Applied Mathematical Sciences\\University of Rhode Island, USA}

\begin{document}
\maketitle

\begin{abstract}
	The Erd\H{o}s--Gallai criteria for recognizing degree sequences of simple graphs involve a system of inequalities. Given a fixed degree sequence, we consider the list of differences of the two sides of these inequalities. These differences have appeared in varying contexts, including characterizations of the split and threshold graphs, and we survey their uses here. Then, enlarging upon properties of these graph families, we show that both the last term and the maximum term of the principal Erd\H{o}s--Gallai differences of a degree sequence are preserved under graph complementation and are monotonic under the majorization order and Rao's order on degree sequences.
\end{abstract}

\section{Erd\H{o}s--Gallai differences} \label{sec: intro}

	Let $d=(d_1,\dots,d_n)$ be the degree sequence of an arbitrary finite, simple graph, and suppose that the terms of $d$ are listed in nonincreasing order. Let $m(d) = \max\{i: d_i \geq i-1\}$; this parameter is called the \emph{modified Durfee number} of $d$. We define the \emph{$k$th Erd\H{o}s--Gallai difference} $\Delta_k(d)$ by \[\Delta_k(d) = k(k-1) + \sum_{i > k} \min\{k,d_i\} - \sum_{i \leq k } d_i,\] for nonnegative integers $k$.
	
	Though we allow $k=0$ (so $\Delta_0(d)=0$) for convenience, we will primarily be interested in the terms of the list \[\Delta(d) = (\Delta_1(d),\dots,\Delta_{m(d)}(d)).\] We call $\Delta(d)$ the list of \emph{principal} Erd\H{o}s--Gallai differences.
	
	The Erd\H{o}s--Gallai differences have appeared implicitly or explicitly in the work of several authors (see, for example, the work of Li~\cite{Li75}, in which the differences appear with the opposite sign). Functioning almost like a ``spectrum'' of a degree sequence, they provide some intriguing pieces of information about graphs, but as yet they seem to have attracted limited notice. In this paper we begin by surveying known properties of $\Delta(d)$. In later sections we prove that $\Delta(d)$ behaves ``nicely'' under the operation of graph complementation (as applied to degree sequences) and under the majorization relation and another degree-sequence-based relation introduced by Rao~\cite{Rao80}.
	
	We organize our survey by the contexts in which the terms of $\Delta(d)$ have previously appeared, briefly summarizing each.
	
	\bigskip
	\noindent
	\textbf{Origins and nonnegativity.} \quad The Er\H{o}s--Gallai differences take their name from a classic results of Erd\H{o}s and Gallai that characterizes the degree sequences of simple graphs. 
	
	\begin{thm}[\cite{ErdosGallai60}] \label{thm: EG criterion}
		A list $d=(d_1,\dots,d_n)$ of nonnegative integers, in nonincreasing order, is the degree sequence of a simple graph if and only if its entries have even sum and \begin{equation} \label{eq: EG ineq}
		\sum_{i \leq k} d_i \leq k(k-1) + \sum_{i > k} \min\{k,d_i\}
		\end{equation}
		holds for all $k \in \{1,\dots,n\}$.
	\end{thm}
	
	Our definition of $\Delta_k(d)$ is simply the difference between the two sides in the inequality~\eqref{eq: EG ineq} with parameter value $k$. Consequently, a sorted list $d$ of $n$ nonnegative integers with even sum is a degree sequence if and only if $\Delta_k(d) \geq 0$ for all $n$.	
	
	\bigskip
	\noindent
	\textbf{Relations to each other.} \quad Authors have improved upon Theorem~\ref{thm: EG criterion} in various ways. One fruitful way weakens the conditions by checking only certain of the inequalities.
	
	\begin{thm}[Li \cite{Li75}; Hammer--Ibaraki--Simeone \cite{HammerIbarakiSimeone81}] \label{thm: Li criterion}
		A sequence $(d_1,\dots,d_n)$ of nonnegative numbers with even sum, arranged in nonincreasing order, is a degree sequence if and only if the $k$th Erd\H{o}s--Gallai inequality holds for all values of $k$ such that $1 \leq k \leq m(d)$, where $m(d) = \max\{i: d_i \geq i-1\}$.
	\end{thm}
	
	Note here the appearance of the modified Durfee number $m(d)$. It is an upper bound on various graph parameters, such as the clique number. Theorem~\ref{thm: Li criterion}'s strengthening of Theorem~\ref{thm: EG criterion} suggests why the terms $\Delta_k(d)$ for $1 \leq k \leq m(d)$ should be the ``principal'' Erd\H{o}s--Gallai differences. When we consider complementary sequences in Section~\ref{sec: complements}, we will see another reason.
	
	The way that Theorem~\ref{thm: Li criterion} and similar criteria (as seen in \cite{BarrusEtAl12, Eggleton75, TripathiVijay03, Zverovich92}, for example) are proved involves showing that certain values of $\Delta(d)$ are bounded by others, so the nonnegativity of $\Delta_k(d)$ need only be checked for certain $k$. Though we do not review all such statements here, we illustrate this idea with the following results of Li.
	\begin{lem}[Li~\cite{Li75}] \label{lem: Li properties}
		Let $d$ be a degree sequence $d=(d_1,\dots,d_n)$ in nonincreasing order, and let $m=m(d)$. 
		\begin{enumerate}
			\item[\textup{(i)}] The differences $\Delta_{m},\Delta_{m+1},\dots,\Delta_n(d)$ form a strictly increasing sequence.
			\item[\textup{(ii)}] If $p$ denotes $\max\{i: d_i \geq m\}$, then the differences $\Delta_{p}(d), \Delta_{p+1}(d),\dots,\Delta_{m}(d)$ are all equal.
			\item[\textup{(iii)}] If $\Delta_q(d)$ is nonnegative, where $q$ is the last index in which the maximum value from $d$ appears, then $\Delta_i(d) \geq 0$ for all $i \in \{1,\dots,\min\{q,m\}\}$.			
		\end{enumerate}
	\end{lem}
	
	\bigskip
	\noindent
	\textbf{Splittance and the last term of $\Delta(d)$.} \quad A split graph is a graph whose vertex set can be partitioned into a clique and an independent set. In~\cite{HammerSimeone81}, Hammer and Simeone defined the \emph{splittance} of a graph $G$ as the minimum number of edges that can be deleted or added to $G$ in order to change it into a split graph. For example, the graph on the left in Figure~\ref{fig: splittance threshold} has splittance 2; if an edge is supplied in place of the dotted nonadjacency, and if the thick edge is removed from the graph, then the resulting graph is split with the black vertices forming a clique and the white vertices forming an independent set.
	
	\begin{figure}
		\centering
		\includegraphics[height=3cm]{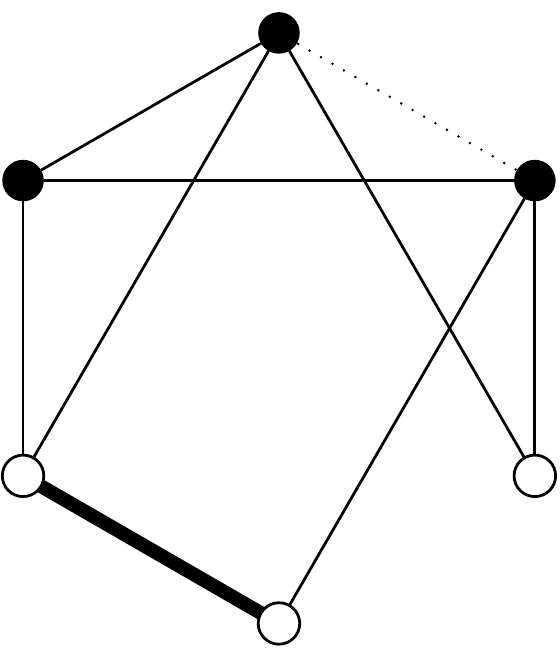} \qquad \qquad \includegraphics[height=3cm]{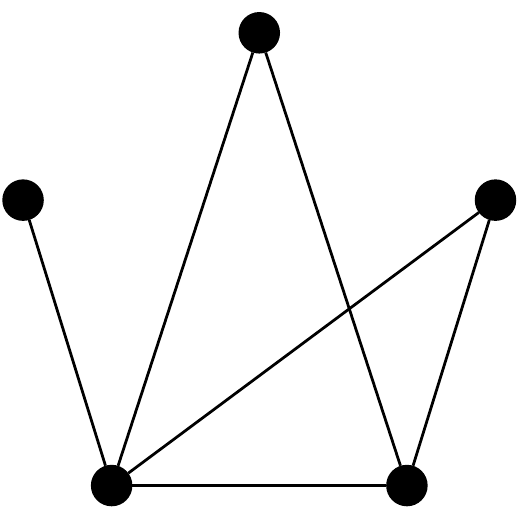}
		\caption{A graph with splittance 2 (left) and a threshold graph (right).}
		\label{fig: splittance threshold}
	\end{figure}
	
	Hammer and Simeone showed that splittance is invariant among all realizations of a degree sequence, and if the degree sequence is $d=(d_1,\dots,d_n)$, with terms written in nonincreasing order, then the splittance is exactly \[\frac{1}{2}\left(m(m-1) + \sum_{i=m+1}^n d_i - \sum_{i=1}^m d_i,\right)\] where $m = m(d)$. Split graphs are precisely the graphs with splittance 0. We restate this result in terms of $\Delta(d)$.
	
	\begin{cor} \label{cor: split iff last EG diff is 0}
		If $d$ is an arbitrary degree sequence $(d_1,\dots,d_n)$, and $G$ is an arbitrary realization of $d$, then $\Delta_{m(d)}(d) = 2s(G)$, where $s(G)$ denotes the splittance of $G$.
		
		Consequently, a graph $G$ having degree sequence $d$ is split if and only if the last term of $\Delta(d)$ equals 0.
	\end{cor}	
	\begin{proof}
		Recall that $\Delta(d)$ has length $m(d)$. Note that though the expression defining splittance varies slightly in form from our definition of $\Delta_{m(d)}(d)$, it is true that $\min\{m,d_i\} = d_i$ for all $i>m$, so $\Delta_m(d)$ is twice the splittance of $G$. 
	\end{proof}	
	
	\bigskip
	\noindent
	\textbf{Graph families with degree sequence characterizations.} \quad The result of Hammer and Simeone just described shows that whether a graph is split depends only on its degree sequence. Another family of graphs where membership is determined by the degree sequence is that of the threshold graphs. As defined by Chv\'atal and Hammer in~\cite{ChvatalHammer77}, a threshold graph is a graph whose characteristic vectors for independent sets can be separated by a hyperplane from those of non-independent sets. Equivalently, a graph $G$ is threshold if and only if there exists a weighting of its vertices and a real threshold such that two vertices in $G$ are adjacent exactly when the sums of their weights meets or exceeds the threshold. For example, the graph on the right in Figure~\ref{fig: splittance threshold} is a threshold graph, since its adjacencies satisfy the definition if each vertex is given weight equal to its degree and the threshold 4.5 is used.

	Threshold graphs have many additional characterizations, in such settings as nested neighborhood conditions, forbidden subgraphs, and the majorization order on degree sequences; see \cite{MahadevPeled95} for a book-length survey including all of these, and see~\cite{HammerKelmans96} for a characterization in terms of Laplacian eigenvalues. The degree sequence characterization was found by Hammer, Ibaraki, and Simeone and is implicit in earlier work by Li; it may be restated in terms of $\Delta(d)$ as the following.
	
	\begin{thm}[\cite{HammerIbarakiSimeone78, HammerIbarakiSimeone81}; see also~{\cite[Theorem 18]{Li75}}] \label{thm: threshold}
		A graph $G$ having degree sequence $d$ is threshold if and only if $\Delta_k(d)=0$ for all $k \in \{1,\dots,m(d)\}$. In other words, $G$ is threshold if and only if every term of $\Delta(d)$ is $0$.
	\end{thm}	
	
	\bigskip
	\noindent
	\textbf{Forced graph structure corresponding to small values in $\Delta(d)$.} \quad Besides the split and threshold graphs, other classes with degree sequence characterizations include the pseudosplit graphs~\cite{MaffrayPreissmann94}, the matroidal and matrogenic graphs~\cite{MarchioroEtAl84,Tyshkevich84}, the hereditary unigraphs~\cite{Barrus13}, and the unigraphs~\cite{KleitmanLi75,Koren76,Li75}. For each of these cases, the characterization relies in part on a (perhaps iterated) decomposition of the vertex set $V(G)$ of a graph $G$ into sets $Q,R,S$ such that the subgraph induced on $Q\cup S$ is a split graph with clique $Q$ and independent set $S$, and each vertex in $R$ is adjacent to every vertex in $Q$ and to no vertex in $S$. This structure is illustrated on the left in Figure~\ref{fig: decomposition, forced} by the solid and dotted lines and curves; note that there is no restriction on edges between $Q$ and $S$. This decomposition, repeatedly carried out until all induced subgraphs involved are indecomposable, yields the \emph{canonical decomposition} of graphs, as defined and extensively studied by Tyshkevich and others (see~\cite{Tyshkevich00} and the references therein).
	
	\begin{figure}
		\centering
		\includegraphics[height=3cm]{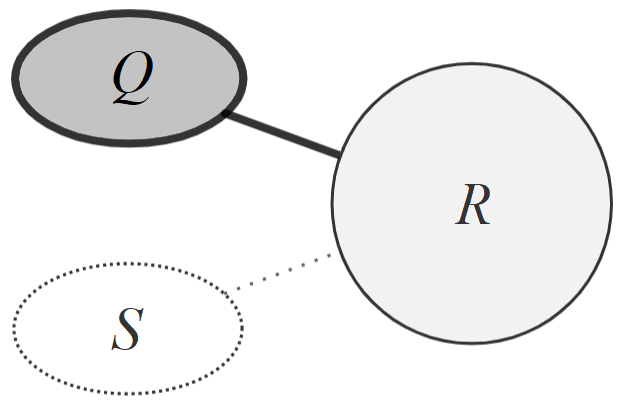} \qquad \qquad \includegraphics[height=2.5cm]{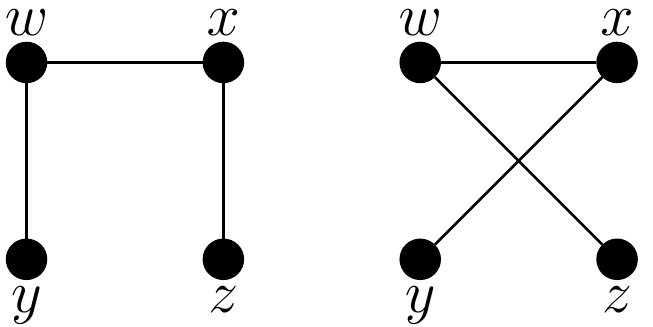}
		\caption{Decomposition corresponding to $\Delta_k(d)=0$ (left) and the labeled realizations of $(2,2,1,1)$ (right).}
		\label{fig: decomposition, forced}
	\end{figure}
	
	This decomposition yields insight into how the principal Erd\H{o}s--Gallai differences are tied to the structure of a graph. One striking connection (and another motivation for focusing attention on the ``principal'' differences) is the following. Here we use $G[X]$ to denote the induced subgraph of $G$ with a given subset $X$ of the vertex set $V(G)$.
	
	\begin{thm}[\cite{Barrus18WT}]
		If $d$ is the degree sequence of $G$, and if $V(G)$ is partitioned into $Q,R,S$ as previously described, then $\Delta(d)$ equals the list formed by appending $\Delta(G[R])$ to $\Delta(G[Q\cup S])$.
	\end{thm}

	There are further connections between $\Delta(d)$ and the canonical decomposition.
	
	\begin{thm}[\cite{Barrus13}]
		A graph $G$ with at least two vertices admits a partition of $V(G)$ into sets $Q,R,S$ as previously described, with $Q\cup S$ and $R$ both nonempty, if and only if one of the following holds.
		\begin{enumerate}
			\item[\textup{(i)}] $G$ has one or more isolated vertices; here we may let $S$ consist of one isolated vertex, with $Q=\emptyset$ and $R = V(G) \setminus (Q \cup S)$.
			\item[\textup{(ii)}] $\Delta_k(d)=0$ for some positive integer $k$; here the vertices in $Q$ may be $k$ vertices of highest degree in $G$, the vertices in $S$ are all those having a degree in $G$ strictly less than $k$, and $R = V(G) \setminus (Q \cup S)$.
		\end{enumerate}
	\end{thm}

	Note that if the partition $Q,R,S$ exists for one realization of a degree sequence $d$, then these sets' same prescribed adjacencies and non-adjacencies (i.e., all possible edges within $Q$, none within $S$, and the edges/non-edges required between $R$ and the other two sets) continue to hold in every other realization of $d$ having the same vertex set and vertex degrees. 
	
	These ideas lead to a new problem. Note that when $\Delta_k(d)=0$ for some $k$, the canonical decomposition leads to certain vertices being guaranteed to be adjacent, or guaranteed to be nonadjacent, in every realization of $d$. We see this on the left in Figure~\ref{fig: decomposition, forced} and more concretely on the right. Here we have the two realizations of $(2,2,1,1)$. These are the only two realizations when the vertices are labeled and $w$ and $x$ are required to have degree $2$ while $y$ and $z$ are required to have degree $1$. Note that $\Delta((2,2,1,1))=(1,0)$. Here $w$ and $x$ are adjacent in each realization, and $y$ and $z$ are nonadjacent in each realization. We say that $w$ and $x$ are ``forcibly'' adjacent, and $y$ and $z$ are ``forcibly'' nonadjacent.
	
	In~\cite{Barrus18Forced, Cloteaux19} the author and Cloteaux independently studied forcible adjacency relationships such as these. Such a phenomenon is not restricted to split graphs like $P_4$ or to graphs with a nontrivial canonical decomposition; for example, in each of the nine labeled realizations of $d=(4,4,3,3,3,1)$, the vertices of degree 4 are adjacent, though the Erd\H{o}s--Gallai difference list $\Delta(d)=(1,1,2,2)$ contains no $0$ term.  The paper~\cite{Barrus18Forced} shows that the occurrence of a forcible adjacency relationship requires that $0 \leq \Delta_k(d) \leq 1$ for some $k$; when some minor technicalities are satisfied, the converse is true as well.
		
	Building off these results, and returning to the theme of special graph families like those mentioned earlier, the paper~\cite{Barrus18WT} \emph{defines} the \emph{weakly threshold graphs} as those graphs having degree sequences $d$ for which $\Delta_k(d) \leq 1$ for all $k$. In addition to their degree sequence characterization, the weakly threshold graphs, like the threshold graphs, split graphs, matroidal/matrogenic graphs, and/or hereditary unigraphs, have characterizations in terms of forbidden induced subgraphs, iterative constructions, and more. These enlarge upon several corresponding results for threshold graphs in pleasing ways.

	\bigskip
	In light of these several contexts in which $\Delta(d)$ appears, we ask whether more useful information about a degree sequence or its realizations may be obtained through study of $\Delta(d)$. In particular, given the results related to split graphs, threshold graphs, and weakly threshold graphs, the final term $\Delta_{m(d)}(d)$ and the maximum term of $\Delta(d)$, which we denote by $\Delta^*(d)$, may be of interest. In this paper, we present results along these lines.
	
	We begin in Section~\ref{sec: M(d)} by presenting matrices that aid in computing and visually representing Erd\H{o}s--Gallai differences.
	
	In Section~\ref{sec: complements} we tackle a problem suggested by the classes of split graphs, threshold graphs, and weakly threshold graphs. Each of these families is closed under complementation, and we prove here that degree sequences of a graph and of its complement always agree on their maximum Erd\H{o}s--Gallai differences and on their final differences.
	
	Finally, in Section~\ref{sec: posets} we show that certain Erd\H{o}s--Gallai differences behave monotonically as one moves upward through one of two posets. We first show that both the last term of $\Delta(d)$ (equivalently, the splittance) and the maximum term are nondecreasing while moving ``upward'' through degree sequences in an induced-subgraph-related poset introduced by Rao~\cite{Rao80}. We then 	improve upon a result from~\cite{Barrus18WT}, in which it was observed that the $k$th Erd\H{o}s--Gallai difference decreases or stays the same as one moves upward through the majorization/dominance poset; here we exactly quantify the changes in $\Delta_k$ during such moves. We see that with this poset as well, both the last and the maximum terms of $\Delta(d)$ exhibit monotonic behavior as one moves up the poset.

\section{The matrix $M(d)$} \label{sec: M(d)}

We preface our results on complements and posets in later sections by presenting two matrices helpful in studying $\Delta(d)$. The first is a version of the \emph{corrected Ferrers diagram}, which has appeared multiple times in degree sequence literature (our presentation follows~\cite{ArikatiPeled94}; see also~\cite{Berge73}, for example). Given a degree sequence $d=(d_1,\dots,d_n)$ with terms in nonincreasing order, define $F(d)$ to be the $n$-by-$n$ matrix having stars on the main diagonal in which for each $i \in \{1,\dots,n\}$ the first $d_i$ non-diagonal entries in the $i$th row are 1, with all other entries set equal to 0. 

When performing matrix computations on a corrected Ferrers diagram $F(d)$ in what follows, we note that the transpose $F(d)^T$ preserves the stars on the main diagonal. We adopt the convention that the numerical value of a star is 0, though for convenience we may continue to write stars instead of 0's on the main diagonal, even after the result of matrix arithmetic.

We now relate the corrected Ferrers diagram to Erd\H{o}s--Gallai differences. As observed in~\cite{Barrus18WT}, for each value $k \in \{1,\dots,m(d)\}$, Erd\H{o}s--Gallai difference $\Delta_k(d)$ is the difference of (i) the number of nonzero terms below the main diagonal of $F$ within the first $k$ columns of $F(d)$ and (ii) the number of nonzero terms to the right of the main diagonal within the first $k$ rows of $F(d)$. We adapt this result with the definition of a second helpful matrix.

\begin{defn}
	Given a degree sequence $d$, define its \emph{(Erd\H{o}s--Gallai) difference matrix} as the matrix $M(d) = F(d)^T - F(d)$, where $\star$ entries (still having a value of 0) continue to occupy the main diagonal.
\end{defn}

\begin{obs} \label{obs: first rows of M(d)}
	For any degree sequence $d$ and value $k \in \{0,\dots,m(d)\}$, the $k$th Erd\H{o}s--Gallai difference $\Delta_k(d)$ is equal to the sum of the entries in the first $k$ rows of $M(d)$.
\end{obs}

We note that the matrix $M(d)$ has the form \[\begin{bmatrix}0 & B\\-B^T & 0\end{bmatrix},\] where the 0 block in the top left is a square block of size $m(d)$; more generally, 0 entries in $M(d)$ occur where both the matrices $F(d)$ and $F(d)^T$ agree. Nonzero entries in the blocks $B$ and $B^T$ are $\pm 1$ and occur in locations where one of $F(d), F(d)^T$ is 0 and the other is 1; the definition of $M(d)$ as $F(d)^T-F(d)$ ensures that $M(d)$ is skew-symmetric.

\begin{exa}
	Suppose that $d = (6,5,3,3,3,1,1,1,1)$. The matrix $F=F(d)$ is shown here:
	\[F = \begin{bmatrix}
	\star & 1 & 1 & 1 & 1 & 1 & 1 & 0 & 0\\
	1 & \star & 1 & 1 & 1 & 1 & 0 & 0 & 0\\
	1 & 1 & \star & 1 & 0 & 0 & 0 & 0 & 0\\
	1 & 1 & 1 & \star & 0 & 0 & 0 & 0 & 0\\
	1 & 1 & 1 & 0 & \star & 0 & 0 & 0 & 0\\
	1 & 0 & 0 & 0 & 0 & \star & 0 & 0 & 0\\
	1 & 0 & 0 & 0 & 0 & 0 & \star & 0 & 0\\
	1 & 0 & 0 & 0 & 0 & 0 & 0 & \star & 0\\
	1 & 0 & 0 & 0 & 0 & 0 & 0 & 0 & \star
	\end{bmatrix}.\]
	Below, we show the difference matrix $M(d) = F^T-F$.
	\[M(d) = \begin{bmatrix}
	\star & 0 & 0 & 0 & 0 &   0 & 0 & 1 & 1\\
	0 & \star & 0 & 0 & 0 & -1 & 0 & 0 & 0\\
	0 & 0 & \star & 0 & 1 & 0 & 0 & 0 & 0\\
	0 & 0 &  0 & \star & 0 & 0 & 0 & 0 & 0\\
	0 & 0 & -1 & 0 & \star & 0 & 0 & 0 & 0\\
	0 & 1 &  0 & 0 & 0 & \star & 0 & 0 & 0\\
	0 & 0 &  0 & 0 & 0 & 0 & \star & 0 & 0\\
	-1 & 0 &  0 & 0 & 0 & 0 & 0 & \star & 0\\
	-1 & 0 &  0 & 0 & 0 & 0 & 0 & 0 & \star
	\end{bmatrix}.\]
	
	Note that $m(d)=4$. Summing the entries in the first $i$ rows for $1 \leq i \leq 4$ shows that $\Delta(d) = (2,1,2,2)$. 
	\QEDA
\end{exa}

In light of Observation~\ref{obs: first rows of M(d)}, the $k$th Erd\H{o}s--Gallai difference is a signed count of asymmetries in a portion of $F(d)$. We will see in Theorem~\ref{thm: sigma(d,i)} that summing the entries in the first $i$ rows for larger $i$ than $m(d)$ also gives useful information. 

We conclude this section with an observation illustrated by the matrix $M(d)$ in the example above.

\begin{lem} \label{lem: mth term equals m-1}
	For any degree sequence $d$ with $m=m(d)$, if $d_m=m-1$, then $\Delta_{m}(d) = \Delta_{m-1}(d)$.
\end{lem}
\begin{proof}
	Under these conditions, the $m$th row of $M(d)$ contains only zeroes, so Observation~\ref{obs: first rows of M(d)} yields the result.
\end{proof}

\section{Complementation} \label{sec: complements}

In this section we show that some of the values appearing in $\Delta(d)$ appear in $\Delta(\overline{d})$, where $\overline{d}$ is the degree sequence of the complement of a realization of $d$. In particular, the maximum and final principal Erd\H{o}s--Gallai differences agree, i.e., $\Delta^*(\overline{d}) = \Delta^*(d)$ and $\Delta_{m(d)}(d) = \Delta_{m(\overline{d})}(\overline{d})$.

These result generalize properties observed in those graph classes mentioned in Section~1 whose membership can be recognized through Erd\H{o}s--Gallai differences. Each of the classes of threshold graphs, weakly threshold graphs, and split graphs is closed under complementation. Proving this is typically done by considering structural characterizations of those graphs or by noting that the set of minimal forbidden induced subgraphs for each class is likewise closed under complementation; see~\cite{Barrus18WT, ChvatalHammer77, FoldesHammer77}. As a consequence of our work here, we obtain a proof based entirely on vertex degrees.

First, we develop a useful matrix operation that has appeared in various places but seems not to have developed a standard notation. Our presentation is adapted from~\cite{GolyshevStienstra07}. Given a matrix $M$, let $M^T$ denote its usual transpose about the main diagonal. We use the notation $M_\bot$ to indicate the transpose of $M$ ``about its antidiagonal''; if the entries of $M$ are denoted by $M_{i,j}$, where $1 \leq i \leq m$ and $1 \leq j \leq n$, then the $(i',j')$-entry of $M_\bot$ is defined by \[\left(M_\bot\right)_{i',j'} = M_{(m+1-j'),(n+1-i')}\] for $1 \leq i' \leq n$ and $1 \leq j' \leq m$. So, for example, we have \[\begin{bmatrix}
1 & 2 & 3\\
4 & 5 & 6
\end{bmatrix}_\bot = \begin{bmatrix}
6 & 3\\
5 & 2\\
4 & 1
\end{bmatrix}.
\]
It is straightforward to verify that the actions of transposing about the main diagonal and the antidiagonal commute with each other, and their composition corresponds to a $180^\circ$ ``rotation'' of the matrix. As with the standard transpose, we also find that $\bot$ distributes over sums of matrices and that $(M_\bot)_\bot = M$ for each matrix $M$. As suggested by~\cite{GolyshevStienstra07}, if $M$ is an $m$-by-$n$ matrix and $H_t$ denotes the $t$-by-$t$ matrix with $H_{i,j} = 1$ if $i+j=t+1$ and $H_{i,j}=0$ otherwise (assuming that $1 \leq i,j \leq t$), then \[M_\bot = H_nM^TH_m.\] This allows us to conclude, for instance, that for any matrices $A,B$ that are $m$-by-$p$ and $p$-by-$n$, respectively, \[(AB)_\bot = H_n(AB)^T H_m = H_n B^T A^T H_m = H_n B^T H_p H_p A^T H_m = B_\bot A_\bot.\]

Turning now to Ferrers diagrams, note that, like $F(d)^T$, the matrix $F(d)_\bot$ preserves the appearance of stars on the main diagonal. 

For any degree sequence $d$, observe that the complementary sequence $\overline{d}$ is equal to $(n-1-d_n,\dots,n-1-d_1)$. The matrix $M(\overline{d})$ also has a nice relationship with $M(d)$.

\begin{lem} \label{lem: M(complement)}
	If $d$ is a degree sequence and $\overline{d}$ is its complementary sequence, then $M(\overline{d}) = M(d)_\bot$.
\end{lem}
\begin{proof}
	The matrix $F(\overline{d})$ is obtained by rotating $F(d)$ by $180^\circ$ and switching non-diagonal entries from 0 to 1, and vice versa. Using matrix operations, we write \[F(\overline{d}) = (J_n - I_n - F(d))^T_\bot,\] where $I_n$ and $J_n$ denote the $n$-by-$n$ identity and all-ones matrices, respectively. We then compute:
	\begin{align*}
		M(\overline{d}) &= F(\overline{d})^T - F(\overline{d})\\
		&= [(J_n - I_n - F(d))^T_\bot]^T - (J_n - I_n - F(d))^T_\bot\\
		&= \left(F(d)^T - F(d)\right)_\bot\\
		&= M(d)_\bot.
	\end{align*}
\end{proof}

\begin{exa}
	Suppose again that $d = (6,5,3,3,3,1,1,1,1)$. We compute $\overline{d} = (7,7,7,7,5,5,5,3,2)$. The matrices $F(\overline{d})$ and $M(\overline{d})$ are as follows:
	\[F(\overline{d}) = \begin{bmatrix}
	\star & 1 & 1 & 1 & 1 & 1 & 1 & 1 & 0\\
	1 & \star & 1 & 1 & 1 & 1 & 1 & 1 & 0\\
	1 & 1 & \star & 1 & 1 & 1 & 1 & 1 & 0\\
	1 & 1 & 1 & \star & 1 & 1 & 1 & 1 & 0\\
	1 & 1 & 1 & 1 & \star & 1 & 0 & 0 & 0\\
	1 & 1 & 1 & 1 & 1 & \star & 0 & 0 & 0\\
	1 & 1 & 1 & 1 & 1 & 0 & \star & 0 & 0\\
	1 & 1 & 1 & 0 & 0 & 0 & 0 & \star & 0\\
	1 & 1 & 0 & 0 & 0 & 0 & 0 & 0 & \star
	\end{bmatrix} = \left(J_9 - I_9 - F(d)\right)^T_\bot\]
	(notice how the 0's ``trace out'' the shape $F(d)$'s 1's, but in ``rotated'' manner);
	\[M(\overline{d}) = \begin{bmatrix}
	\star &  0 & 0 & 0 &  0 & 0 & 0 &  0 & 1\\
	0 &  \star & 0 & 0 &  0 & 0 & 0 &  0 & 1\\
	0 &  0 & \star & 0 &  0 & 0 & 0 &  0 & 0\\
	0 &  0 & 0 & \star &  0 & 0 & 0 & -1 & 0\\
	0 &  0 & 0 & 0 &  \star & 0 & 1 &  0 & 0\\
	0 &  0 & 0 & 0 &  0 & \star & 0 &  0 & 0\\
	0 &  0 & 0 & 0 & -1 & 0 & \star &  0 & 0\\
	0 &  0 & 0 & 1 &  0 & 0 & 0 &  \star & 0\\
	-1 & -1 & 0 & 0 &  0 & 0 & 0 &  0 & \star
	\end{bmatrix} = M(d)_\bot.\]
	Here $m(\overline{d})=6$, and summing the entries in the first rows of $M(\overline{d})$ shows that $\Delta(\overline{d}) = (1,2,2,1,2,2)$.
	\QEDA
\end{exa}

As we will see, these connections between matrices related to $d$ and to $\overline{d}$ imply that the Erd\H{o}s--Gallai differences agree on certain values. We begin with a brief lemma that is not original (see \cite[Proposition 16]{Li75} for a similar result) but is included here for completeness.

\begin{lem} \label{lem: complementary m}
	Each degree sequence $(d_1,\dots,d_n)$ and its complementary sequence $(\overline{d}_1,\dots,\overline{d}_n) = (n-1-d_n,\dots,n-1-d_1)$ satisfy \[n \leq m(d) + m(\overline{d}) \leq n+1.\] Equality holds in the latter inequality if and only if $d_m = m-1$.
\end{lem}
\begin{proof}
	Abusing notation, let $m=m(d)$ and $\overline{m} = m(\overline{d})$. Observe that \[\overline{d}_{n-m} = n-1 - d_{n+1-(n-m)} = n-1 - d_{m+1} > n-1-m;\] hence $\overline{m} \geq n-m$.
	
	On the other hand, \begin{multline*}\overline{d}_{n-m+2} \leq \overline{d}_{n-m+1} = n-1 - d_{n+1-(n-m+1)} \\ = n-1 - d_{m} \leq n-1-(m-1) < (n-m+2) - 1;\end{multline*} thus $\overline{m} \leq n-m+1$.
	
	Now if $m+\overline{m} = n$, then $\overline{d}_{n-m+1} < n-m$, so $n-1-d_{m} < n-m$ and thus $d_{m} > m-1$. If $m+\overline{m} = n+1$, then $\overline{d}_{n+1-m} \geq n-m$ and $n-1-d_m \geq n-m$ and so $d_{m} \leq m-1$; since $d_m \geq m-1$, this implies that $d_m = m-1$.
\end{proof}

\begin{thm} \label{thm: sigma(d,i)}
	Let $d$ be a degree sequence. For $i \in \{0,\dots,n\}$, let $\sigma(d,i)$ denote the sum of all entries in the first $i$ rows of $M(d)$. Then \[\sigma(d,i) = \begin{cases}
	\Delta_i(d) & \text{if } 0 \leq i \leq m(d);\\
	\Delta_{n-i}(\overline{d}) & \text{if } n-m(\overline{d}) \leq i \leq n.
	\end{cases}\]
\end{thm}

\begin{proof}
	The first case on the right-hand side follows from Observation~\ref{obs: first rows of M(d)}.
	
	For $0 \leq k \leq n$, let $h_k$ denote the $1$-by-$n$ vector in which the first $k$ entries are 1 and all other entries are 0. Let $\mathbf{1}$ denote the $n$-by-$1$ all-ones vector. For convenience let $I = I_n$ and $J = J_n$. Now observe that for $1 \leq k \leq m(\overline{d})$, since $\Delta_k(\overline{d})$ is a scalar and $M(d)$ is a skew symmetric matrix,
	\begin{align*}
	\Delta_k(\overline{d}) &= h_kM(\overline{d})\mathbf{1}\\
	&= h_kM(d)_\bot\mathbf{1}\\
	&= (\mathbf{1}^T - h_{n-k})^T_\bot M(d)_\bot \mathbf{1}^T_\bot\\
	&= [(\mathbf{1}^T - h_{n-k}) M(d)^T \mathbf{1}]^T_\bot\\
	&= (\mathbf{1}^T - h_{n-k}) M(d)^T \mathbf{1}\\
	&= \mathbf{1}^T M(d)^T \mathbf{1} - h_{n-k} M(d)^T \mathbf{1}\\
	&= 0 - h_{n-k} M(d)^T \mathbf{1}\\
	&= h_{n-k} M(d) \mathbf{1}.
	\end{align*}
	Hence the $k$th Erd\H{o}s--Gallai difference of $\overline{d}$ is the sum of the entries in all but the last $k$ rows of the matrix $M(d)$; this yields the second case on the right-hand side, and the last case follows since $M(d)$ is skew-symmetric.
\end{proof}

Theorem~\ref{thm: sigma(d,i)} shows that the sequence of terms $\sigma(d,k)$ lists the Erd\H{o}s--Gallai differences of $d$, in order, transitioning in later terms to list the Erd\H{o}s--Gallai differences of $\overline{d}$ in reverse order.

\begin{exa}
	To conclude the example from Section~\ref{sec: M(d)}, the terms $\sigma(d,k)$ for $d=(6,5,3,3,3,1,1,1,1)$ are, beginning with $\sigma(d,0)$, \[0,2,1,2,2,1,2,2,1,0;\] note that the initial terms of the sequence are $\Delta_0(d)$, followed by the principal Erd\H{o}s--Gallai differences for $d$ (i.e., $2,1,2,2$), while the differences for $\overline{d}$, which were $1,2,2,1,2,2$, show up in reverse order before the final term $\Delta_0(\overline{d})$. 
	\QEDA
\end{exa}

That the final principal Erd\H{o}s--Gallai differences for $d$ and $\overline{d}$ overlap in the example above is no coincidence, and it illustrates our first result on common differences for $d$ and $\overline{d}$.
\begin{cor} \label{cor: mth difference}
	If $d$ is a degree sequence and we let $m=m(d)$ and $\overline{m} = m(\overline{d})$, then 
	$\Delta_{\overline{m}}(\overline{d}) = \Delta_{m}(d)$.
\end{cor}
\begin{proof}
	If $m + \overline{m} = n$, then $n-\overline{m} = m$, and Theorem~\ref{thm: sigma(d,i)} yields $\Delta_{\overline{m}}(\overline{d}) = \sigma(d,n-\overline{m}) = \sigma(d,m) = \Delta_{m}(d)$.
	
	If $m + \overline{m} = n+1$, then $d_m=m-1$, so Lemma~\ref{lem: mth term equals m-1} and Theorem~\ref{thm: sigma(d,i)} yield $\Delta_{\overline{m}}(\overline{d}) = \sigma(d,m-1) = \Delta_{m-1}(d) = \Delta_{m}(d)$.	
\end{proof}

As we now show, there are often additional values that appear among the principal Erd\H{o}s--Gallai differences of both $d$ and $\overline{d}$. These arise from ``islands'' of nonzero entries in $M(d)$. In the following, use \emph{line} to refer to any row or column of a matrix, and recall that star entries are treated as zero entries.

\begin{lem} \label{lem: islands}
	The following statements hold for nonzero entries in $M(d)$.
	\begin{enumerate}
		\item[\textup{(i)}] If two entries in a single line of $M(d)$ are both nonzero, then they and all entries between them in that line are equal.
		\item[\textup{(ii)}] $M(d)$ contains no submatrix of the form $\begin{bmatrix}a & b\\0 & c\end{bmatrix}$, where $a$ and $c$ are nonzero.
	\end{enumerate}
\end{lem}
\begin{proof}
	In the difference matrix $M(d)=F(d)^T-F(d)$, a nonzero entry occurs in a position where one of $F(d)$ or $F(d)^T$ has a zero entry and the other does not. Both statements (i) and (ii) follow from the fact that in any line of $F(d)$ or $F(d)^T$, each 1 entry precedes any 0 entries. 
\end{proof}

\begin{lem} \label{lem: complement diff equality}
	For $k \in \{1,\dots,m(d)\}$, suppose that no nonzero entry in the $(k+1)$th row of $M(d)$ lies in the same column as a nonzero entry in the $k$th row. If the leftmost nonzero entry among the first $k$ rows appears in the $j$th column of $M(d)$, then $\Delta_{n+1-j}(\overline{d}) = \Delta_{k}(d)$. 
\end{lem}
\begin{proof}
	Note that under the given hypotheses we have $m(d) < j \leq n$, which together with Lemma~\ref{lem: complementary m} yields $1 \leq n+1-j \leq n-m(d) \leq m(\overline{d})$.
	
	Now by hypothesis and Lemma~\ref{lem: islands}, the nonzero entries in the first $k$ rows of $M(d)$ lie in columns that contain only 0 entries following the $k$th entry. Hence the sum of the first $k$ rows of $M(d)$, which equals $\Delta_k(d)$, is also equal to the sum of the columns containing these nonzero entries, which also equals the sum of the entries in each of the columns $j,\dots,n$ in $M(d)$, by Lemma~\ref{lem: islands}(ii). This is precisely the sum of the entries in the first $n+1-j$ rows of $M(d)_\bot$, and since $M(d)_\bot = M(\overline{d})$, we conclude that $\Delta_k(d) = \Delta_{n+1-j}(\overline{d})$.
\end{proof}

\begin{thm} \label{thm: max EG diff}
	Let $d$ be a degree sequence, and let $\overline{d}$ be its complementary sequence. If one of the following is true, then $\Delta_k(d)$ is one of the principal Erd\H{o}s--Gallai differences of $\overline{d}$:
	\begin{enumerate}
		\item[\textup{(i)}] $k=m(d)$;
		\item[\textup{(ii)}] $k=1$ and $\Delta_1(d) > \Delta_2(d)$;
		\item[\textup{(iii)}] $k \in \{2,\dots,m(d)-1\}$ and it is \emph{not} true that $\Delta_{k-1}(d) < \Delta_{k}(d) <\Delta_{k+1}(d)$ or that $\Delta_{k-1}(d) > \Delta_{k}(d) > \Delta_{k+1}(d)$.
	\end{enumerate}
	
	In particular, the maximum principal Erd\H{o}s--Gallai differences coincide, i.e., $\Delta^*(\overline{d}) = \Delta^*(d)$.
\end{thm}
\begin{proof}
	The case $k=m(d)$ was shown in Corollary~\ref{cor: mth difference}.
	
	In each of the other cases, if $\Delta_k(d) = v$ for some $v$, it suffices to assume that $k$ is lowest index making $\Delta_k(d)=v$ and satisfying the conditions of the theorem. This guarantees that the $k$th row of the difference matrix $M(d)$ contains a nonzero entry. 
	
	Suppose first that $\Delta_k(d) = \Delta_{k+1}(d)$. Then the $(k+1)$th row of $M(d)$ must contain only 0 entries, so the hypotheses of Lemma~\ref{lem: complement diff equality} are met, and $\Delta_k(d)$ appears as a principal Erd\H{o}s--Gallai difference of $\overline{d}$.
	
	If $\Delta_k(d) > \Delta_{k+1}(d)$, then by our assumptions either $k=1$ or $\Delta_{k-1}(d) < \Delta_k(d)$, so the $k$th row of $M(d)$ contains positive entries. We also see that the $(k+1)$th row of $M(d)$ contains no positive entries. By Lemma~\ref{lem: islands}, this implies that no nonzero entry in the $(k+1)$th row of $M(d)$ lies in the same column as a nonzero entry in the $k$th row, so the hypotheses of Lemma~\ref{lem: complement diff equality} are met, and we again have the desired conclusion.
	
	If $\Delta_k(d) < \Delta_{k+1}(d)$, then $\Delta_{k-1}(d) > \Delta_k(d)$, so the $k$th row of $M(d)$ contains negative entries while the $(k+1)$th row does not. Again Lemma~\ref{lem: islands} implies that hypotheses of Lemma~\ref{lem: complement diff equality} are met, and the claim holds.
	
	Finally, observe that if $\Delta_k(d)= \Delta^*(d)$, then one of the conditions (i)--(iii) holds, so $\Delta^*(d)$ appears among the principal Erd\H{o}s--Gallai differences of $\overline{d}$, and 
	$\Delta^*(\overline{d}) \geq \Delta^*(d)$. Interchanging the roles of $d$ and $\overline{d}$ shows that equality holds.
\end{proof}

\begin{cor}[{Compare with~\cite{Barrus18WT, ChvatalHammer77, FoldesHammer77}}]
	The classes of split, threshold, and weakly threshold graphs are closed under complementation.	More generally, for any set $I$ of nonnegative integers, the following classes of graphs are closed under complementation:
	\begin{align*}
	\mathcal{A}_I &= \left\{G\;:\; G \textit{ is a simple graph whose degree sequence } d \text{ satisfies } \Delta_{m(d)}(d) \in I\right\};\\
	\mathcal{B}_I &= \left\{G\;:\; G \textit{ is a simple graph whose degree sequence } d \text{ satisfies } \Delta^*(d) \in I\right\};\\	\end{align*}
\end{cor}
\begin{proof}
	The general statement follows from Corollary~\ref{cor: mth difference} and Theorem~\ref{thm: max EG diff}. By Corollary~\ref{cor: split iff last EG diff is 0} and Theorem~\ref{thm: threshold}, the families of split and threshold graphs are respectively the families $\mathcal{A}_I$ and $\mathcal{B}_I$ when $I=\{0\}$. The family of weakly threshold graphs is the family $\mathcal{B}_I$ when $I=\{0,1\}$.
\end{proof}

\section{Two posets} \label{sec: posets}

Here we show how values of principal Erd\H{o}s--Gallai differences compare for degree sequences related under either of two partial orders important to the study of degree sequences. In particular, we will show that both the parameters of last and largest principal Erd\H{o}s--Gallai difference are monotone under these orders.

\subsection{The Rao order}

In~\cite{Rao80}, S.B.~Rao introduced a partial order on degree sequences in the following way: for degree sequences $e$ and $d$, define $e \preccurlyeq d$ if there exist realizations $H$ of $e$, and $G$ of $d$, such that $H$ is an induced subgraph of $G$. Rao showed that $\preccurlyeq$ does indeed yield a partial order (we call it $\mathcal{R}$) and posed some questions about this order. In the decades since the poset's introduction, the most celebrated result about it has been the proof by Chudnovsky and Seymour that $\mathcal{R}$ is a well quasi-order~\cite{ChudnovskySeymour14}. As such, if a family of degree sequences forms an ideal (i.e., a downward-closed subposet) in $\mathcal{R}$, then it can be characterized by finitely many minimal obstructions. One example of such a characterization has appeared in~\cite{Barrus14}; another will in~\cite{BarrusTrenkWhitman2?}.

How do the Erd\H{o}s--Gallai differences behave with respect to the order relation $\preccurlyeq$ ? Consider the example where $e=(2,2,2,1,1)$ and $d=(4,3,3,2,2,2)$. Here $e \preccurlyeq d$ in $\mathcal{R}$, and $\Delta(e) = (2,2,2)$ while $\Delta(d) = (1,3,2)$. Note that entrywise, a term in $\Delta(d)$ may be equal to, less than, or greater than the corresponding term in $\Delta(e)$. For other pairs $e,d$ it may also be the case that $\Delta(d)$ may have more terms than $\Delta(e)$. However, a few relationships between the terms of $\Delta(e)$ and $\Delta(d)$ always hold, as we now show.

Consider first the last terms of $\Delta(e)$ and $\Delta(d)$. Recall from Section~\ref{sec: intro} that the splittance of a degree sequence $d$ (or of a graph $G$ realizing $d$) is the minimum number of edges to be deleted from or added to $G$ in order to transform it into a split graph. This number is the same for any realization of $d$, and it is equal to one-half the value of $\Delta_{m(d)}(d)$.

\begin{thm}
	If degree sequences $e,d$ satisfy $e \preccurlyeq d$ in $\mathcal{R}$, then $\Delta_{m(e)}(e) \leq \Delta_{m(d)}(d)$.
\end{thm}
\begin{proof}
	For any degree sequence $\pi$, let $s(\pi)$ denote the splittance of any realization of $\pi$; for a graph $P$, let $s(P)$ denote the splittance of $P$. Given that $e \preccurlyeq d$ in $\mathcal{R}$, let $H$ and $G$ respectively be realizations of $e$ and $d$ such that $H$ is an induced subgraph of $G$. Observe that \[\Delta_{m(e)}(e) = 2s(e) = 2s(H) \leq 2s(G) = 2s(d) = \Delta_{m(d)}(d),\] since any adjacency modifications that transform $G$ into a split graph necessarily transform its induced subgraph $H$ as well.
\end{proof}

Now consider the maximum Erd\H{o}s--Gallai differences of $e$ and $d$.

\begin{thm}
	If degree sequences $e,d$ satisfy $e \preccurlyeq d$ in $\mathcal{R}$, then $\Delta^*(e) \leq \Delta^*(d)$.
\end{thm}
\begin{proof}
	It suffices by induction to assume that $d$ has one more term than $e$ (since a realization $H$ of $d$ that contains an induced subgraph $G$ that is a realization of $e$ can be built up by restoring one vertex of $V(H) \setminus V(G)$ at a time to $G$; the degree sequences of the intermediate graphs lie between $e$ and $d$ in $\mathcal{R}$). Suppose that $d$ is the degree sequence resulting from adding a new vertex $v$ to a realization $G$ of $e$. If $a$ is the degree of $v$, then $d$ is obtained by increasing $a$ terms of $e$ each by $1$, inserting a term equal to $a$ in the list, and reordering the list terms as necessary so that $d$ is in nonincreasing order.
	
	Equivalently, $F(d)$ is obtained from $F(e)$ by
	\begin{enumerate}
		\item[(i)] adding one more row of zeroes at the bottom of the matrix and one more column of zeroes at the right side of $F(e)$, with a star in the entry on the main diagonal,
		\item[(ii)] changing the first non-diagonal 0 term of a row into a 1, in each of $a$ distinct rows, and
		\item[(iii)] changing the first non-diagonal 0 term of a column into a 1, in each of the \emph{first} $a$ columns.
	\end{enumerate}
	
	Now let $k$ be an index such that $\Delta_k(e) = \Delta^*(e)$. Let $n$ be the length of $e$. Since \[\sum_{\substack{i \leq k\\1 \leq j \leq n+1}}(F(d)^T)_{ij} = \sum_{\substack{i \leq k\\1 \leq j \leq n}} (F(e)^T)_{ij} + \min\{k,a\}\] and \[\sum_{\substack{i \leq k\\1 \leq j \leq n+1}} (F(d))_{ij} \leq \sum_{\substack{i \leq k\\1 \leq j \leq n}} (F(e))_{ij} + \min\{k,a\},\]
	
	The value $\Delta_k(d)$, which by Observation~\ref{obs: first rows of M(d)} equals the difference of the two left-hand sides in the expressions above, is necessarily at least as large as the difference of the two right-hand sides above. Hence $\Delta^*(d) \geq \Delta_k(d) \geq \Delta_k(e) = \Delta^*(e)$.
\end{proof}

\subsection{The dominance order}

In this section we consider the \emph{dominance order} $\mathcal{D}_s$. The ground set for this partially ordered set consists of all degree sequences having a fixed sum $s$ (which equals twice the number of edges in realizations of the sequences). Two degree sequences $d,e$ in this set satisfy $d \succeq e$ if $d$ \emph{majorizes} $e$, that is, if \[\sum_{i \leq k} d_i \geq \sum_{i \leq k} e_i\] for all integers $k$, assuming that the terms of $d$ and $e$ are respectively in nonincreasing order.

Classic results on $\mathcal{D}_s$ and the relation $\succeq$ include the fact that $\mathcal{D}_s$ is an ideal, or ``downward-closed'' subposet when the majorization order is applied to all partitions (not just degree sequences) of a fixed even positive number; this was proved by Ruch and Gutman~\cite{RuchGutman79}. In other words, for any partitions $d$, $e$ of even positive integer $s$, if $d \succeq e$ and $d$ is the degree sequence of a simple graph, then $e$ is as well. On the other hand, as shown in~\cite{Merris03,Barrus18WT,RuchGutman79}, if degree sequences $d$ and $e$ with the same sum satisfy $d \succeq e$ and $e$ is respectively split, weakly threshold, or threshold, then $d$ must be as well.

A key notion in studying $\mathcal{D}_s$ is the \emph{unit transformation}, which is an operation performed on degree sequences $d$ in $\mathcal{D}_s$. If the terms of $d$ are $d_i$, indexed so that $d_1 \geq d_2 \geq \cdots$, a unit transformation on $d$ subtracts 1 from some term $d_r$ and adds 1 to $d_t$ for some $t$ such that $d_r \geq d_t+2$ (if $d_r \geq 2$, we also permit ourselves to imagine that a 0 immediately following nonzero terms of $d$ can be augmented to become 1, even if $d$ was not previously assumed to have any terms equal to 0). A well known result (sometimes called \emph{Muirhead's Lemma}, as in~\cite{Merris03}) states that $d \succeq e$ if and only if $e$ can be obtained by a finite sequence of unit transformations on $d$.

We examine the effect of a single unit transformation on $\Delta(d)$, making use of the matrices $F(d)$ and $M(d)$ introduced in Section~\ref{sec: M(d)}. As above, suppose that $e$ is obtained from a unit transformation on $d$, and $r,t$ are the indices where the modifications of the terms of $d$ happen, so \[e_i = \begin{cases}
d_i-1 & \text{if } i=r;\\
d_i+1 & \text{if } i=t;\\
d_i & \text{otherwise};
\end{cases}\]
and $r<t$. We further suppose that $t$ is the first index following $r$ such that $d_t \leq d_r - 2$ while $r$ is the last index preceding $t$ such that $d_r \geq d_t+2$; this ensures that $d_r > d_{r+1}$, that  $e$'s terms remain in nonincreasing order, and that $e$ cannot be produced from $d$ through a composition of two or more unit transformations (as may happen if $d_t < d_u \leq d_r-2$ for some $u$, for instance). For convenience, assume that $d$ and $e$ both have $n$ terms, where a trailing 0 is added to $d$ if necessary.

We may write $F(e) = F(d) + C_{rt}(d)$ for some $n$-by-$n$ matrix $C_{rt}(d)$ having a single $-1$ in Row $r$ and a $1$ in Row $t$; let $j_r$ and $j_t$ respectively denote the indices of the columns of $C_{rt}(d)$ where these nonzero entries occur. Every other entry of $C_{rt}(d)$ is equal to 0. Note that the star entries in $F(d)$ cause 
\begin{equation} \label{eq: jr jt}
j_r = \begin{cases}
d_r + 1 & \text{if } d_r \geq r;\\
d_r     & \text{if } d_r < r;
\end{cases} \qquad \text{and} \qquad j_t = \begin{cases}
d_t + 2 & \text{if } d_t \geq t-1;\\
d_t + 1 & \text{if } d_t < t-1.
\end{cases}
\end{equation}
We claim that $j_t < j_r$. Indeed, since $d_r \geq d_t+2$, this is immediate unless $d_t+2 = j_t = j_r = d_r$; however, this exceptional possibility cannot happen, since it requires both that $d_t \geq t-1$ and $d_r < r$, causing \[d_r < r < t \leq 1+d_t \leq d_r-1.\] 

Observe that $M(e) = M(d) + C'_{rt}(d)$, where $C'_{rt}(d) = C_{rt}(d)^T - C_{rt}(d)$. The matrix $C'_{rt}(d)$ is skew-symmetric and has either two or four nonzero entries: there are entries of $+1$ in positions $(r,j_r)$ and $(j_t,t)$, or an entry of $+2$ if these positions coincide, and entries equal to $-1$ in positions $(t,j_t)$ and $(j_r,r)$, with an entry of $-2$ instead if these positions coincide. Observe that no location of a positive entry in $C'_{rt}(d)$ is in danger of coinciding with the location of a negative entry, since $r<t$ and $j_t < j_r$ imply that $(r,j_r)\neq (t,j_t)$ and $(j_t,t) \neq (j_r,r)$, and \eqref{eq: jr jt} shows that $j_r \neq r$ and $j_t \neq t$, so $(r,j_r) \neq (j_r,r)$ and $(t,j_t) \neq (j_t,t)$.

Using Observation~\ref{obs: first rows of M(d)} and the fact that $M(e)=M(d)+C'_{rt}(d)$, the description of $C'_{rt}(d)$ just given lets us trace the effect on the Erd\H{o}s--Gallai differences due to unit transformations. 

\begin{lem}\label{lem: Delta(d) to Delta(e)}
	If $d$ and $e$ satisfy the assumptions given previously, then \[\Delta_k(e) = \begin{cases}
		\Delta_k(d) + c_1(d,r,t,k) - c_{-1}(d,r,t,k) & \text{if } k \leq \min\{m(d),m(e)\};\\
		\Delta_{m(d)}(d) + 2 & \text{if } k = m(e) > m(d),
		 \end{cases}\] where $c_1(d,r,t,k)$ is the number of elements (counting multiplicities) from $\{r,j_t\}$ that are less than or equal to $k$, and $c_{-1}(d,r,t,k)$ is the number of elements (counting multiplicities) from $\{t,j_r\}$ that are less than or equal to $k$.	
\end{lem}
\begin{proof}
	The first case was proved above; the second case can only happen when \[d_{m(d)+1} < m(d) \leq e_{m(d)+1} \leq d_{m(d)+1}+1;\] in this situation $m(e)=m(d)+1$. Here $e_{m(d)+1} = m(d)$, so by Lemma~\ref{lem: mth term equals m-1}, $\Delta_{m(e)}(e) = \Delta_{m(d)}(e)$. Since we further know that $(t,j_t) = (m(d)+1,d_{m(d)+1}+1)$ and $r \leq m(d)$ while $j_r>m(d)+1$, we see that $c_1(d,r,m(d)+1,m(d)) = 2$ and $c_{-1}(d,r,m(d)+1,m(d)) = 0$, and the first case in our formula yields the result.	
\end{proof}

Lemma~\ref{lem: Delta(d) to Delta(e)} provides exact values for differences $\Delta_k(e)$ as unit transformations occur; as consequences, we have the following previously known results.

\begin{cor}\label{cor: majorization facts} \mbox{}
	\begin{enumerate}
		\item[\textup{(1)}] (See~\cite{Barrus18WT}) If $e$ is obtained by a unit transformation on $d$ and $m' = \min\{m(d), m(e)\}$, then for any $k \in \{1,\dots,m'\}$ we have $\Delta_k(d) \leq \Delta_k(e)$.
		\item[\textup{(2)}] (Li~\cite{Li75}; see also~\cite{Barrus18WT} and Corollary~\ref{cor: split iff last EG diff is 0} herein) For any degree sequence $d$, $\Delta_{m(d)}(d)$ is an even number.		
	\end{enumerate}
\end{cor}
\begin{proof}
	(1): Since $r<t$ and $j_t<j_r$,  it follows that $c_1(d,r,t,k) \geq c_{-1}(d,r,t,k)$.
	
	(2): Because of the definition of $m(d)$, for $u \in \{r,t\}$ it follows that $u \leq m(d)$ if and only if $j_u > m(d)$. Thus $c_1(d,r,t,m(d)) - c_{-1}(d,r,t,m(d))$ has either the form $2-0$ or $1-1$ for any $d$. Theorem~\ref{thm: threshold} implies that any threshold graph's degree sequence $\pi$ satisfies $\Delta_{m(\pi)}(\pi)=0$. Since every degree sequence is majorized by a threshold sequence~\cite{RuchGutman79}, Lemma~\ref{lem: Delta(d) to Delta(e)} and an inductive argument show that $\Delta_{m(d)}(d)$ is even.
\end{proof}

Our next lemma shows that the last term of $\Delta(d)$ is monotone with respect to the dominance order. 

\begin{lem} \label{lem: last term is nondecreasing}
	If $d$ and $e$ are degree sequences with a common sum such that $d\succeq e$, then $\Delta_{m(d)}(d) \leq \Delta_{m(e)}(e)$. 
\end{lem}
\begin{proof}
	By Muirhead's Lemma, it suffices to prove the result in the case that some unit transformation changes $d$ into $e$. Suppose that this is so.
	
	If $m(e)=m(d)$ or $m(e)>m(d)$, then Corollary~\ref{cor: majorization facts} or Lemma~\ref{lem: Delta(d) to Delta(e)} gives the result. If instead $m(e)<m(d)$, then the unit transformation changing $d$ to $e$ must change a 1 to a 0 in the $(m(d)-1)$th column of the $m(d)$th row in $F(d)$. This implies that $d_{m(d)} = m(d)-1$, which by Lemma~\ref{lem: mth term equals m-1} implies that $\Delta_{m(d)}(d) = \Delta_{m(e)}(e)$.
\end{proof}

This gives context to the following result of Merris that degree sequences of split graphs (those $d$ for which the last term of $\Delta(d)$ is 0) appear at the ``top'' of $\mathcal{D}_s$.

\begin{cor}[{\cite[Lemma 3.3]{Merris03}}]
	If $d$ and $e$ are degree sequences such that $d \succeq e$ and $e$ is the degree sequence of a split graph, then $d$ is the degree sequence of a split graph as well.
\end{cor}
\begin{proof}
	This follows from Lemma~\ref{lem: last term is nondecreasing} and Corollary~\ref{cor: split iff last EG diff is 0}.
\end{proof}

We now extend the monotonicity to general indices and beyond individual unit transformations.

\begin{thm} \label{thm: monotone diffs in dominance order}
	If $d$ and $e$ are degree sequences with a common sum and $d \succeq e$, then $\Delta_k(d) \leq \Delta_k(e)$ for all $k \in \{1,\dots, m'\}$, where $m' = \min\{m(d),m(e)\}$. 
\end{thm}
\begin{proof}
	Let $\pi^0, \pi^1,\dots,\pi^p$ denote a sequence of degree sequences satisfying $\pi^0 = d$, $\pi^p = e$, and for each $i \in \{1,\dots,p\}$ the degree sequence $\pi^i$ is obtained by performing a unit transformation on $\pi^{i-1}$. Suppose that $k \in \{1,\dots,m'\}$. If $k \leq m(\pi^i)$ for all $i$ in $\{0,\dots,p\}$, then $\Delta_k(d) \leq \Delta_k(e)$ follows inductively by Corollary~\ref{cor: majorization facts}. It thus suffices to prove the conclusion in the case that $m(\pi^0)=m(\pi^p) = k$ and $m(\pi^i)<k$ for all $i$ such that $0<i<p$. 
	
	Under this assumption, let $m_i$ denote $m(\pi^i)$ for $i \in \{0,\dots,p\}$.  Lemma~\ref{lem: last term is nondecreasing} implies that 
	\[\Delta_k(\pi^0) = \Delta_{m_0}(\pi^0) \leq \Delta_{m_1}(\pi^1) \leq \dots \leq \Delta_{m_p}(\pi^p) = \Delta_k(\pi^p),\]
	and our claim is proved.
\end{proof}

To conclude the paper, we return to the maximum term among Erd\H{o}s--Gallai difference lists.

\begin{thm} \label{thm: monotonic max diff}
	If $d$ and $e$ are degree sequences satisfying $d \succeq e$, then $\Delta^*(d) \leq \Delta^*(e)$.
\end{thm}
\begin{proof}
	This is immediate from the previous theorem if for some index $k$ we have $\Delta_k(d) = \Delta^*(d)$ and $k \leq m(e)$. If instead $\Delta(d)$ achieves the value $\Delta^*(d)$ at some index $k$ that is greater than $m(e)$, then consider a sequence $\pi^0, \pi^1,\dots,\pi^p$ of degree sequences satisfying $\pi^0 = d$, $\pi^p = e$, and for each $i \in \{1,\dots,p\}$ the degree sequence $\pi^i$ is obtained by performing a unit transformation on $\pi^{i-1}$. Since $m(d)>m(e)$ and the length of the Erd\H{o}s--Gallai difference list of a degree sequence can change by at most 1 during a unit transformation, at some term $\pi^\ell$ in the sequence, the length of $\Delta(\pi^\ell)$ is equal to $k$, and each of $\Delta(\pi^0),\dots,\Delta(\pi^\ell)$ has length at least $k$. By Corollary~\ref{cor: majorization facts} we see that $\Delta_{m(\pi^\ell)}(\pi^\ell) \geq \Delta^*(d)$. Then by Lemma~\ref{lem: last term is nondecreasing}, applied iteratively to the sequence $\pi^\ell,\dots,\pi^p = e$, we see that $\Delta^*(e) \geq \Delta^*(d)$, as claimed.	
\end{proof}

\begin{cor}[{\cite{RuchGutman79,Barrus18WT}}]
	If $d$ and $e$ are degree sequences satisfying $d \succeq e$, where $e$ is the degree sequence of a threshold (respectively, weakly threshold) graph, then $d$ is also the degree sequence of a threshold (resp., weakly threshold) graph.
\end{cor}
\begin{proof}
	This follows from Theorem~\ref{thm: monotonic max diff} and Theorem~\ref{thm: threshold} and the definition of weakly threshold graphs.
\end{proof}

\end{document}